\documentclass[12pt]{article}
\usepackage[centertags]{amsmath}
\usepackage{amsfonts}
\usepackage{amssymb}
\usepackage{amsthm}

\input{amssym.def}

\newtheorem{Definition}{Definition}[section]
\newtheorem{Theorem}[Definition]{Theorem}
\newtheorem{Lemma}[Definition]{Lemma}

\newtheorem{Corollary}[Definition]{Corollary}

\newtheorem{Example}[Definition]{Example}

\newcommand{\lc}{\mathcal{L}}
\newcommand{\rc}{\mathcal{R}}
\newcommand{\hc}{\mathcal{H}}
\newcommand{\jc}{\mathcal{J}}

\title{\Large \bf On inverse ordered semigroups}
\author{A. Jamadar and K. Hansda\\
\footnotesize{Department of Mathematics, Visva-Bharati, University}\\
\footnotesize{Santiniketan, Bolpur-731235, West Bengal, India}\\
\footnotesize{kalyanh4@gmail.com},\
\footnotesize{amlanjamadar@gmail.com}}

\begin{document}
\date{}
\maketitle

\begin{abstract}{\footnotesize}
The purpose of this paper is to study the generalization of
inverse semigroups (without order). An ordered semigroup $S$ is
called an inverse ordered semigroup if for every $a\in S$, any two
inverses of $a$ are $\hc$-related. We prove that an ordered
semigroup is complete semilattice of t-simple ordered semigroups
if and only if it is completely regular and inverse. Furthermore
characterizations of inverse ordered semigroups have been
characterized by their ordered idempotents.
\end{abstract}

\section{Introduction}

 An ordered semigroup  is a partiality ordered set
$(S,\leq)$, and at the same time a semigroup $(S,\cdot)$ such that
for all $a, b, x \in S,  \;a \leq b$ implies
 $xa\leq xb \;\textrm{and} \;a x \leq b x $.
It is denoted by $(S,\cdot, \leq)$. For every subset $H$ of $S$,
the downword closure subset of $H$ is denoted by $(H]$ and defined
by $(H]=\{t\in S: t\leq h, \;\textrm{for some} \;h\in H\}$.
Throughout this paper unless otherwise stated $S$ stands for an
ordered semigroup. T. Saito \cite{Saito 1971} studied inverse
semigroups by introducing partial order in it. There was much
interest in the past decade in studying the inverses of an element
in an ordered semigroup. Bhuniya and Hansda \cite{bh1} prove that
for every $a$ in a regular ordered semigroup $S$ and for two
inverses $x, y$ of $a$, $x\hc y$ if and only if for all $e,f\in
E_\leq (S)$, $ef\in (fSe]$. Thus it is interesting to further
study ordered semigroup in which any two inverses of an element
are $\hc$-related. Class of these ordered semigroups are natural
generalization of class of inverse semigroups (without order).  We
call these ordered semigroups as inverse ordered semigroups. This
paper is inspired by Clifford and Preston \cite{Clifford}.

\section{Preliminaries}

An equivalence relation $\rho$ is called left (right) congruence
if for $a, b, c \in S \;a\rho b \;\textrm{ implies} \;ca \rho cb
\;(ac \rho bc)$.  By a congruence we mean both left and right
congruence. A congruence $\rho$ is called semilattice congruence
on $S$ if  for all $a, b \in S, \;a \rho a^{2} \;\textrm {and}
\;ab \rho \;ba$. By a complete semilattice congruence on $S$ we
mean a semilattice congruence $\sigma$ on $S$ such that for $a, b
\in S, \;a \leq b$ implies that $a \sigma ab$. An ordered
semigroup $S$ is called complete semilattice of subsemigroups of
type $\tau$ if there exists a complete semilattice congruence
$\rho $ such that $(x)_{\rho}$ is a type $\tau$ subsemigroup of
$S$. Let $I$ be a nonempty subset of $S$. Then $I$ is called a
left(right) ideal of $S$, if $SI\subseteq I (IS\subseteq I)$ and
$(I]\subseteq I$. An ideal $I$ is both  a left and a right ideal
of $S$. We call $S$ a (left, right) simple ordered semigroup if it
does not contain any proper (left, right) ideal. Following
Kehayopulu \cite{Ke2006} principal left ideal of $S$ generated by
$a$ is defined by the set $L(a)= \{xa: x\in S^1\}$, and principal
ideal generated by $a$ is defined by the set $I(a)= \{xay: x, y\in
S^1\}$.

$S$ is said to be regular (resp. Completely regular, right
regular) ordered semigroup if for every $a \in S, \;a\in
(aSa](a\in (a^2Sa^2], \;a\in (a^2S])$. Due to Kehayopulu
\cite{Ke2006} Green's relations on a regular ordered semigroup
given as follows:

$a\lc b\; \textrm{if} \;L(a)= L(b)$, $a\rc b \;\textrm{if} \;R(a)=
R(b)$, $a\jc b \;\textrm{if} \;I(a)= I(b)$, $\hc=\lc \cap \rc$.

This four relation $\lc, \rc, \jc,\textrm{and}  \;\hc$ are
equivalence relation.

A regular ordered semigroup $S$ is said to be group-like (resp.
left group-like) \cite{bh1} ordered semigroup if for every $a,b\in
S, \;a\in (Sb] \;\textrm{and} \;b \in (aS](\textrm{resp}. \;a \in
(Sb])$. Right group-like ordered semigroup can be defined dually.
An element $b\in S$ is said to be an inverse of $a\in S$ if $a\leq
aba$ and $b\leq bab$. The set of all inverses of an element $a$ is
denoted by $V_\leq (a)$. By an ordered idempotent of  $S$, we mean
an element $e\in S$ such that $e\leq e^2$. The set of all ordered
idempotents of $S$ are denoted by $E_\leq (S)$. Any two elements
$a,b \in S$ are said to be $\hc$-commutative if $ab\leq bxa$ for
some $x\in S$.

For the sake of convenience of general reader we state some
results of \cite{bh1}.
\begin{Lemma}\cite{bh1}\label{1}
Let $S$ be a completely regular ordered semigroup. Then following
statements hold in $S$:
\begin{enumerate}
\item\vspace{-.4cm} For every $a$ there is $x\in S$ such that
$a\leq axa^2$ and $a\leq a^2xa$.
\item\vspace{-.4cm}  $\jc$ is the
least complete semilattice congruence on $S$.
\item\vspace{-.4cm}
$S$ is a complete semilattice of completely simple ordered
semigroups.
\end{enumerate}
\end{Lemma}

\section{ Inverse ordered semigroup}
Let $S$ be an ordered semigroup and $\rho$ be an equivalence
relation on $S$. We call an ideal $I$ of $S$ is generated by an
$\rho$-unique element $b\in S$ if for any generator of $x\in I$,
$b\rho x$.
\begin{Definition}
A regular ordered semigroup $S$ is called inverse if  for every
$a\in S$, any two inverses of $a$ are $\hc$-related.
\end{Definition}
\begin{Example}
The ordered semigroup $S= \{a, e, f  \}$ defined by multiplication
and order below is an inverse ordered semigroup.
\begin{center}
\begin{tabular}{|l|l|l|l|}
  \hline
  $\cdot$ & $a $&$e$ & $f$ \\
  \hline
  $a$ & $a$ & $e$ & $f$ \\
  \hline
  $e$ & $f$ & $e$ & $a$ \\
  \hline
  $f$ & $e$ & $a $ & $f$ \\
  \hline
\end{tabular}
\end{center}
$$'\leq ' := \{(a,a),  (e,e),  (f,f)\}.$$

\end{Example}

We present a  role of ordered idempotents in an inverse ordered
semigroup in the next  theorem.
\begin{Theorem}\label{7} An ordered semigroup $S$ is
inverse if and only if every principal left ideal and every
principal right ideal of $S$ are generated by an $\hc$-unique
ordered idempotent.
\end{Theorem}
\begin{proof}
Suppose that $S$ is inverse. Let $I$ be a principal left ideal of
$S$. Then there exists $e\in E_\leq (S)$ such that $I= (Se]$. If
possible let $I= (Sf]$ for some $ f \in E_\leq(S)$. Then $e\lc f$
and thus $e\leq xf$ and $f\leq ye$ for some $x, y\in S$. Now
$e\leq ee\leq eee\leq exfe$. Therefore $exf\leq exfexf$ so that
$exf\in E_\leq(S)$. Also $exf\leq exfexf\leq exf(fe)exf $ and
$fe\leq feee\leq fexfe\leq fe(exf)fe $. Therefore $fe\in V_\leq
(exf)$. Also $ exf\in V_\leq (exf)$. Since $S$ is inverse, we have
 $fe\hc exf$. Then $e\leq ee\leq exf.fe\leq fezexf$ for some $z\in S$, and so $e\leq fz_1$, where $z_1= ezexf$. Similarly $f\leq ez_2$ for some $z_2\in S$. So $e\rc
f$. Hence $e\hc f$. Likewise every principal right ideal of $S$
generated by $\hc$-unique ordered idempotent.

Conversely assume that given conditions hold in $S$. Let $a\in S$
and $a', a''\in V_\leq (a)$. Clearly $(Sa]= (Sa'a]= (Sa''a]$.
Since $a'a, a''a\in E_\leq (S)$ we have that $a'a\hc a''a$, by
given condition. Then there are $s, t\in S$ such that $a'\leq
a''asa'$ and $a''\leq a'ata''$. Thus $a'\rc a''$. Likewise $a'\lc
a''$, that is $a'\hc a''$. Hence $S$ is an inverse ordered
semigroup.
\end{proof}
In the following we show  that an ordered semigroup $S$ is inverse
if and only if any two ordered idempotents of $S$ are
$\hc$-commutative.
\begin{Theorem}\label{5}
The following conditions are equivalent on an ordered semigroup
$S$.
\begin{enumerate}
\item\vspace{-.4cm} $S$ is an inverse semigroup;
\item\vspace{-.4cm} $S$ is regular and its idempotents are
$\hc$-commutative; \item\vspace{-.4cm} For every $e,f\in
E_\leq(S)$, $e\lc f(e\rc f)$ implies $e\hc f$.
\end{enumerate}
\end{Theorem}
\begin{proof}
$(1)\Rightarrow(2)$: Obviously $S$ is regular. Let us assume that
$a\in S$ and $a', a''\in V_\leq(a)$. Consider $e, f\in E_\leq(S)$.
Since $S$ is regular, so there is $x\in S$ such that $x\in
V_\leq(ef)$. Now $x\leq xefx$ implies that $fxe\leq fxe(ef)fxe$
and $ef\leq efxef$ implies $ef\leq ef(fxe)ef$. Thus $ef\in
V_\leq(fxe)$. Also $fxe\leq fxefxe$ that is $fxe\in E_\leq (S)$.
So $fxe\in V_\leq(fxe)$. Since $S$ is inverse, so $fxe\hc ef$.
Then there are $s_1, s_2 \in S$ such that $ef\leq fxes_1$ and
$ef\leq s_2fxe$. Now $ef\leq efxef$ implies that $ef\leq
f(xes_1xs_2fx)e$. Therefore $ef\leq fye$, where $y= xes_1xs_2fx$.
Similarly there is $z\in S$ such that $fe\leq ezf$. Hence any two
idempotents are $\hc$-commutative.

$(2)\Rightarrow (3)$: Let $e,f\in E_\leq (S)$ be such that $ e\lc
f$. Then $e\leq xf$ and $f\leq ye$ for some $x,y\in S$. Now $e\leq
xf $ implies $e\leq exf$, and so $e\leq ee\leq exfe$ which implies
that $exf\leq exfexf$. So $exf\in E_{\leq}(S)$. Similarly $fye\in
E_{\leq}(S)$. Now $e\leq exf\leq exff \leq exffye$. Since $exf,
fye\in E_\leq(S)$, by condition (2) we have $exffye\leq
(fye)z(exf)$ for some $z\in S$ . Hence $e\leq ft$, where
$t=yezexf$. Similarly $f\leq ew$ for some $w\in S$, so that $e\rc
f$. Hence $e\hc f$. If $ e\rc f$ then $e\hc f$ can be done dually.

$(3)\Rightarrow (1)$: Let $a\in S$ and $a', a'' \in V_{\leq}(a)$.
Now $aa'\leq aa''aa'$ and $aa''\leq aa'aa''$. So $aa'\rc aa''$
which implies that $aa'\hc aa''$, by the condition (3). Also
$a'a\hc a''a$. Then $a'\leq a'aa'$ gives that $a'\leq a''axa$ for
some $x\in S$. Therefore $a'\leq a''t$ where $t= axa$. In similar
way it is possible to obtained $u,v, w\in S$ such that $a'\leq
ua''$, $a''\leq a'v$ and $a''\leq wa'$. So $a'\hc a''$. Hence $S$
is an inverse ordered semigroup.

\end{proof}
\begin{Lemma}
Let $S$ be an inverse ordered semigroup. Then following statements
hold in $S$.
\begin{enumerate}
\item\vspace{-.4cm}
 $a\lc b$ if and only if $a'a\hc b'b$ for
some $a, b\in S$ and $a'\in V_{\leq}(a)$ $b'\in V_{\leq}(b)$;
\item\vspace{-.4cm}
 $a\rc b$ if and only if $aa'\hc bb'$ for some $a, b\in S$ and $a'\in V_{\leq}(a)$ $b'\in V_{\leq}(b)$;
 \item\vspace{-.4cm}
 for any $a\in S$ and $e\in E_{\leq}(S)$ there are $x, y \in S $ such that $aexa', a'eya\in E_{\leq}(S)$; where $a'\in V_{\leq}(a)$.
\item\vspace{-.4cm} for any $a, b \in S$ there are $x, y\in S$
such that $ab\leq abb'xa'ab$ and $b'a'\leq b'a'aybb'a'$, where
$a'\in V_{\leq}(a)$ and $b'\in V_{\leq}(b)$.
\end{enumerate}
\end{Lemma}
\begin{proof}
(1): Let $a,b\in S$ be such that $a\lc b$. Let $a'\in
V_{\leq}(a)$, $b'\in V_{\leq}(b)$. Since $a\leq aa'a$ and $a'a\leq
a'aa'a$, we have $a\lc a'a$ which implies that $b\lc a'a$. Also
$b\lc b'b$. Hence $a'a\lc b'b$. Since $a'a, b'b\in E_{\leq}(S)$
and  $S$ is inverse we have $a'a\hc b'b$, by Theorem \ref{5}(3).

Conversely suppose that given condition holds in $S$. Let $a, b
\in S$ with $a'\in V_\leq (a)$ and $b'\in V_\leq (b)$. Then by
given condition $aa'\hc bb'$. Also we have $a\lc a'a$ and $b\lc
b'b$ so that $a\lc b$.

(2): This is similar to (1).

(3): Let $a\in S$ and $e\in E_{\leq}(S)$. Also $a'a\in
E_{\leq}(S)$. Since $S$ is an inverse, there is an $x\in S$ such
that $a'ae\leq exa'a$ by Theorem \ref{5}(2). Now $aexa'\leq
aa'aeexa'\leq aexa'aexa'$. So $aexa'\in E_{\leq}(S)$. Likewise
$a'eya\in E_{\leq}(S)$; for some $y\in S$.

(4): Let $a,b \in S$ with  $a'\in V_{\leq}(a)$, $b'\in
V_{\leq}(b)$. So $a'a, b'b\in E_\leq (S)$. Now $ab\leq
aa'abb'b\leq$ and $a'abb'\leq b'bxa'a$, by Theorem \ref{5}(2).
Thus $ab\leq abb'xa'ab$. Likewise $b'a'\leq b'a'aybb'a'$; for some
$y\in S$.
\end{proof}
In the following theorem an inverse ordered semigroup has been
characterized by the inverse of an element of the set $(eSf]$.

\begin{Theorem}
Let $S$ be an ordered semigroup and $e, f \in E_\leq (S)$. Then
$S$ is inverse if and only if for every $x\in (eSf]$ implies
$x'\in (fSe]$, where $x'\in V_\leq (x)$.
\end{Theorem}
\begin{proof}

First suppose that $S$ is an inverse ordered semigroup and $x\in
(eSf]$. Then $x\leq es_1f$ for some $s_1\in S$. Let $x'\in V_\leq
(x)$. Now $x'\leq x'xx'\leq x'es_1fx'$, and so $es_1fx'\leq
es_1fx'es_1fx'$. Hence $es_1fx'\in E_\leq (S)$. Similarly
$x'es_1f\in E_\leq (S)$. Now there is $s_2\in S$ such that
$x'es_1fx'\leq x'es_1ffx'\leq fs_2x'es_1fx'$, by Theorem
\ref{5}(2) . Also $fs_2x'es_1fx'\leq fs_2x'ees_1fx'\leq
fs_2x'es_1fx's_3e$,  for some $s_3\in S$. Then $x'\leq x'xx'$
implies that $x'\leq fs_2x'es_1fx'\leq fs_2x'es_1fx's_3e$. Hence
$x'\in (fSe]$.

Conversely assume that the given conditions hold in $S$. First
consider a left ideal $L$ of $S$ such that $L= (Se]= (Sf]$ for $e,
f\in E_\leq (S)$. Then $e\lc f$, so that $e\leq ee\leq ezf$ for
some $z\in S$. Therefore $e\in (eSf]$. Since $e\in V_\leq (e)$ we
have $e\in (fSe]$, by given condition. Likewise $f\in (eSf]$. This
implies that $e\rc f$ and so $e\hc f$. Similarly it can be shown
that every principal right ideal of $S$ generated by $\hc$-unique
ordered idempotent. Thus by Theorem \ref{7}, $S$ is an inverse
ordered semigroup.
\end{proof}

\begin{Corollary}
The following conditions are equivalent on a regular ordered
semigroup $S$.
\begin{enumerate}
\item\vspace{-.4cm}
 $S$ is an inverse ordered semigroup;
 \item\vspace{-.4cm}
 for any $a\in S$ and for any $a'\in V_\leq(a)$, $aa', a'a$ are
 $\hc$-commutative;
 \item\vspace{-.4cm}
 for any $e\in E_\leq (S)$, any two inverses of $e$ are
 $\hc$-related;
 \item\vspace{-.4cm}
 for any $e\in E_\leq (S)$ and all its inverses are
 $\hc$-commutative;
 \item\vspace{-.4cm}
 for any $e\in E_\leq (S)$ and $e'\in V_\leq (e)$, $ee'$ and $e'e$
 are $\hc$-commutative.
\end{enumerate}
\end{Corollary}
\begin{proof}
$(1)\Rightarrow (2)$, $(2)\Rightarrow (3)$, $(3)\Rightarrow (4)$,
and $(4)\Rightarrow (5)$: These are obvious.

 $(5)\Rightarrow (1)$: Let $e, f\in E_\leq(S)$ and $x\in V_\leq
 (ef)$. So $ef\leq efxef\leq effxeef$ and $x\leq xefx$ implies that
 $fxe\leq fxeeffxe$. So $ef\in V_\leq (fxe)$. Also $fxe\in E_\leq
 (S)$. Now $ef\leq efxef\leq effxeef\leq effxefxeef\leq fxez_1efz_2fxe$,
 for some $z_1, z_2\in S$, by the given condition. So $ef\leq fz_3e$ where
 $z_3= xemefnfx$. Similarly  $fe\leq ez_4f$, for some
 $z_4\in S$. So $e, f$ are $\hc$-commutative. Hence by Theorem \ref{5} $S$ is inverse
 ordered semigroup.
\end{proof}
We study  inverse ordered semigroup together with  complete
regularity  in the following theorem.
\begin{Theorem}
The following conditions are equivalent on a regular ordered
semigroup $S$.

\begin{enumerate}
\item\vspace{-.4cm} $S$ is inverse and completely regular;
\item\vspace{-.4cm} $S$ is a complete semilattice of group like
ordered semigroups; \item\vspace{-.4cm}
 $ab\hc ba$ whenever $ab, ba \in E_{\leq}(S)$;
\item\vspace{-.4cm} any ordered idempotent of $S$ is
$\hc$-commutative to any element of $S$; \item\vspace{-.4cm}
 for any $e,f \in E_{\leq}(S)$ $e\jc f$ implies $e\hc f$;
\item\vspace{-.4cm}
 $\hc=\lc=\rc=\jc$.
\end{enumerate}
\end{Theorem}
\begin{proof}
$(1)\Rightarrow (2)$: Let $S$ be a completely regular and inverse
ordered semigroup. Then  by Theorem \ref{1},  $\jc$ is the
complete semilattice congruence on $S$ and every $\hc$-class is a
group-like ordered semigroup. We now prove $\hc= \jc$. Let $a,
b\in S$ be such that $a\jc b$. So there are $x, y, u, v\in S$ such
that $a\leq xby$ and $b\leq uav$. Since $S$ is completely regular,
so there are $h, g, f\in S$ such that $x\leq x^2hx$, $b\leq
b^2gb$, $b\leq bgb^2$, $y\leq yfy^2$. Now $a\leq x^2hxb^2gbyfy^2
\leq x^2hxb^2gbgb^2yfy^2$. Let $p\in V_\leq (x^2hxb^2g)$. So
$x^2hxb^2g\leq x^2hxb^2gpx^2hxb^2g\leq
x^2hxb^2g(b^2gpx^2h)x^2hxb^2g$ and $b^2gpx^2h\leq
b^2gpx^2hxb^2gpx^2h\leq b^2gpx^2h(x^2hxb^2g)b^2gpx^2h$. This shows
that $b^2gpx^2h\in V_\leq (x^2hxb^2g)$. Also $x^2hxb^2g\leq
x^2hxb^2gpx^2hxb^2g\leq x^2hxb^2g(x^2hxb^2gp^2)x^2hxb^2g$ and
$x^2hxb^2gp^2\leq x^2hxb^2gpx^2hxb^2gp^2\leq
x^2hxb^2gp^2(x^2hxb^2g)x^2hxb^2gp^2$, which implies that
$x^2hxb^2gp^2\in V_\leq (x^2hxb^2g)$. Similarly $p^2x^2hxb^2g\in
V_\leq (x^2hxb^2g)$. Since $b^2gpx^2h,  x^2hxb^2gp^2\in V_\leq
(x^2hxb^2g)$ and $S$ is inverse, so there is $t\in S$ such that
$x^2hxb^2gp^2\leq b^2gpx^2ht$. Thus $x^2hxb^2g\leq
x^2hxb^2gpx^2hxb^2g\leq x^2hxb^2gp^2(x^2hxb^2g)^2$ implies that
$x^2hxb^2g\leq b^2gpx^2hxt(x^2hxb^2g)^2= bs$ where $s=
bgpx^2ht(x^2hxb^2g)^2$. Similarly there is $s_1\in S$ such that
$b^2gyfy^2\in s_1b$. Hence $a\leq x^2hxb^2gbyfy^2\leq bsbyfy^2=
bs_2$, where $s_2= sbyfy^2$. Similarly $a\leq s_3b$ for some
$s_3\in S$. Likewise $b\leq s_4a$ and $b\leq as_5 $, for some
$s_4, s_5\in S$. So $a\hc b$. Thus $\jc\subseteq \hc$. Also
$\hc\subseteq \jc$, and Hence $\jc= \hc$. Therefore $S$ is
complete semilattice of group-like ordered semigroups.

 $(2)\Rightarrow (3)$: Suppose that  $S$
is a complete semilattice $Y$ of group like ordered semigroups $
\{S_\alpha\}_{\alpha\in Y}$. Let $a, b\in S$ such that $ab, ba \in
E_{\leq}(S)$. Let $\rho$ be the corresponding semilattice
congruence on $S$. Then there is $\alpha \in Y$ such that $ab,ba
\in S_\alpha$ . Since $S_\alpha$ is
 group like ordered semigroups so $ab\hc ba$.

$(3)\Rightarrow (4)$: Let $a\in S$ and $e\in E_{\leq}(S)$. Since
$S$ is regular there is an $x\in S$ such that $a\leq axa$. Clearly
$ax,xa\in E_{\leq}(S)$. Thus by condition (3) $ax\hc xa$. So
$xa\leq axu$ and $ax\leq vxa$, for some $u, v\in S$. Then
 we have $a\leq axa\leq axaxa\leq axaxaxa\leq aaxuxvxaa=a^2ta^2$, where
$t=xuxvx$. Now $a\leq a^2ta^2\leq a(a^2ta^2ta^2ta^2)a\leq
a^2(a^2ta^2ta^2ta^2ta^2)a$, that is $a\leq a^2ya$, where $y=
a^2ta^2ta^2ta^2ta^2$. Similarly $a\leq aya^2$. Clearly $a^2y,
ya^2\in E_\leq (S)$.

Let $e,f \in E_{\leq}(S)$ and $x\in V_{\leq}(ef)$. Then we have
$x\leq xefx$. So $fxe\leq fxefxe\leq fxeeffxe$ and $ef\leq
efxef\leq effxeef$. So $ef\in V_{\leq}(fxe)$. Also $ef\leq
effxeef$ implies that $effxe\leq effxeeffxe$, and $fxeef\leq
fxeeffxeef$. So $effxe, fxeef\in E_{\leq}(S)$ and thus $effxe\hc
fxeef$, by the condition(3). Then there are $u,v\in S$ such that
$effxe\leq fxeefu$ and $fxeef\leq veffxe$. Now $ef\leq
effxefxeef\leq fxeefuveffxe= fce$; where $c=xe^2fuvef^2x$.
Likewise $fe\leq edf$, for some $d\in S$.

Now $ae\leq a^2yae$. Let $z\in V_\leq (a^2yae)$. So $a^2yae\leq
a^2yaeza^2yae\leq a^2yae(eza^2y)a^2yae$. Clearly $a^2yaeeza^2y,
eza^2ya^2yae\in E_\leq (S)$ and thus $a^2yaeeza^2y\hc
eza^2ya^2yae$, by condition (3). Now $ae\leq a^2yae\leq
a^2yaeeza^2ya^2yae\leq eza^2ys_1a^2yaea^2yae$, for some $s_1\in
S$. So $ae\leq es_2ae$, where $s_2= za^2ys_1a^2yaea^2y$. Again
$ae\leq es_2 aya^2e\leq es_2aes_3ya^2 $, for some $s_3\in S$,
since $ya^2, e\in E_\leq (S)$. That is $ae\leq es_4a$, for some
$s_4\in S$. Similarly $ea\leq as_5e$, for some $s_5\in S$. So $a,
e$ are $\hc$-commutative.

$(4)\Rightarrow (5)$: Let $e,f\in E_{\leq}(S)$ such that $e\jc f$.
Then there are $x, y, z, u\in S$ such that $e\leq xfy$ and $f\leq
zeu$. Now $e\leq xfy$ implies that $e\leq fhxy$ and $e\leq xykf$
by the given condition for some $h,k \in S$. Similarly $f\leq zeu$
gives $f\leq es_1zu$ and $f\leq zus_2e$ for some $s_1, s_2 \in S$.
Hence $e\hc f$.

$(5)\Rightarrow (6)$: Let $a, b \in S$ such that $a\jc b$. Then
there are $s,t,u,v \in S$ such that $a\leq sbt$ and $b\leq uav$.
Since $S$ is regular so $a\leq axa$ and $b\leq byb$ for some $x,
y\in S$ so that $ax\leq axax$ and $by\leq byby$. Now $axax\leq
axsbtx\leq axsbybtx$ that is $ax\leq axsbybtx$. Likewise $by\leq
byuaxavy $. Thus $ax\jc by$, so from given condition $ax\hc by$.
Similarly $xa\hc yb$. So  there is $c\in S$ such that $ax\leq
byc$, that is $a\leq byca= bd$, for some $d=yca\in S$. Likewise
$a\leq pb$, $b\leq qa$ for some $p,q\in S$. Thus $a\hc b$. So
$\hc=\jc$. Now $\jc=\hc=\lc\cap \rc$ gives $\jc\subseteq \lc$ and
$\jc\subseteq \rc$. Therefore $\lc=\jc=\rc$.

$(6)\Rightarrow (1)$: Let $a\in S$ Since $S$ is regular so there
exists $a'\in V_\leq (a)$. Clearly $a\lc a'a$ and $a\rc aa'$. So
by the given condition $a\rc a'a$ and $a\lc aa'$. Now $a\leq
aa'a\leq aa'aa'a\leq aa'aa'aa'a\leq aas_1a's_2aa$ for some $s_1,
s_2\in S$. So $a\leq a^2pa^2$ where $p= s_1a's_2$. So $S$ is
completely regular.

Also let $a', a''\in V_\leq (a)$. Now $a\lc a'a \lc a''a$ implies
that $a\rc a'a\rc a''a$. Also by the given condition we can show
that $a\lc aa'\lc a''a$. So it is to check that $a'\rc a''$ and
$a'\lc a''$. So $a'\hc a''$. Hence $S$ is inverse ordered
semigroup.
\end{proof}

\bibliographystyle{plain}

\end{document}